\documentclass[a4paper, 10pt, parskip=half]{scrartcl}

\usepackage[utf8]{inputenc}
\usepackage[T1]{fontenc}
\usepackage{lmodern}
\usepackage{amsmath}
\usepackage{amssymb}
\usepackage{amsthm}
\usepackage{amsfonts}
\usepackage{dsfont}
\usepackage{mathtools}
\usepackage{marginnote}
\usepackage{xcolor}
\usepackage{hyperref}
\hypersetup{
  colorlinks=false,
  pdfborder={0 0 0},
  pdftitle={Finite-time blow-up in a two-dimensional Keller--Segel system with an environmental dependent logistic source},
  pdfauthor={Mario Fuest},
  pdfkeywords={chemotaxis, critical mass, finite-time blow-up, logistic source},
  bookmarksopen=true,
}

\RequirePackage{geometry}
\newcommand{\R}{\mathbb{R}}

\newcommand{\N}{\mathbb{N}}

\newcommand{\ur}[1]{\mathrm{#1}}
\newcommand{\ure}{\ur e}

\newcommand{\eps}{\varepsilon}

\newcommand{\gt}{>}
\newcommand{\lt}{<}

\newcommand{\defs}{\coloneqq}
\newcommand{\sfed}{\eqqcolon}

\newcommand{\ra}{\rightarrow}

\newcommand{\nea}{\nearrow}

\newcommand{\sea}{\searrow}

\newcommand{\ol}{\overline}

\newcommand{\diff}{\,\mathrm{d}}

\newcommand{\ds}{\,\mathrm{d}s}

\newcommand{\dr}{\,\mathrm{d}r}

\newcommand{\dsigma}{\,\mathrm{d}\sigma}

\newcommand{\drho}{\,\mathrm{d}\rho}

\newcommand{\ddt}{\frac{\mathrm{d}}{\mathrm{d}t}}

\newcommand{\hp}{\hphantom}
\newcommand{\pe}{\mathrel{\hp{=}}}

\newcommand{\tmax}{T_{\max}}
\newcommand{\tmmax}{\hat T_{\max}}
\newcommand{\intom}{\int_\Omega}
\newcommand{\intns}{\int_0^{s_0}}
\newcommand{\ombar}{\ol \Omega}

\newcommand{\leb}[1]{{L^{#1}(\Omega)}}
\newcommand{\sob}[2]{{W^{#1, #2}(\Omega)}}
\newcommand{\con}[1]{{C^{#1}(\ombar)}}

\renewcommand{\paragraph}[2][.]{\textbf {#2#1}}

\newtheoremstyle{nplain}
  {\topsep}   
  {\topsep}   
  {\itshape}  
  {0pt}       
  {\bfseries} 
  {.}         
  {5pt plus 1pt minus 1pt} 
  {\thmnumber{#2 }\thmname{#1}\thmnote{ (#3)}} 

\newtheoremstyle{ndefinition}
  {\topsep}   
  {\topsep}   
  {}   	      
  {0pt}       
  {\bfseries} 
  {.}         
  {5pt plus 1pt minus 1pt} 
  {\thmnumber{#2 }\thmname{#1}\thmnote{ (#3)}} 

\newtheorem{base}{Base}[section]
\numberwithin{equation}{section}

\theoremstyle{nplain}
\newtheorem{theorem}[base]{Theorem} \newtheorem*{theorem*}{Theroem}
\newtheorem{lemma}[base]{Lemma} \newtheorem*{lemma*}{Lemma}
\newtheorem{prop}[base]{Proposition} \newtheorem*{prop*}{Proposition}
 \newtheorem*{cor*}{Corollary}

\theoremstyle{ndefinition}
 \newtheorem*{definition*}{Definition}
 \newtheorem*{example*}{Example}
 \newtheorem*{cond*}{Condition}

\newtheorem{remark}[base]{Remark} \newtheorem*{remark*}{Remark}

\begin{document}
\setkomafont{title}{\normalfont \Large}
\title{Finite-time blow-up in a two-dimensional Keller--Segel system with an environmental dependent logistic source}
\author{
Mario Fuest\footnote{fuestm@math.uni-paderborn.de}\\
{\small Institut f\"ur Mathematik, Universit\"at Paderborn,}\\
{\small 33098 Paderborn, Germany}
}
\date{}
\maketitle

\begin{abstract}
  \noindent
  The Neumann initial-boundary problem for the chemotaxis system
  \begin{align} \label{prob:abstract} \tag{$\star$}
    \begin{cases}
      u_t = \Delta u - \nabla \cdot (u \nabla v) + \kappa(|x|) u - \mu(|x|) u^p, \\
      0   = \Delta v - \frac{m(t)}{|\Omega|} + u, \quad m(t) \defs \intom u(\cdot, t)
    \end{cases}
  \end{align}
  is studied in a ball $\Omega = B_R(0) \subset \R^2$, $R \gt 0$
  for $p \ge 1$ and sufficiently smooth functions $\kappa, \mu: [0, R] \ra [0, \infty)$.\\[5pt]
  We prove that whenever $\mu', -\kappa' \ge 0$ 
  as well as $\mu(s) \le \mu_1 s^{2p-2}$ for all $s \in [0, R]$ and some $\mu_1 \gt 0$
  then for all $m_0 \gt 8 \pi$ there exists $u_0 \in \con0$ with $\intom u_0 = m_0$
  and a solution $(u, v)$ to \eqref{prob:abstract} with initial datum $u_0$ blowing up in finite time.
  If in addition $\kappa \equiv 0$ then  all solutions with initial mass smaller than $8 \pi$ are global in time,
  displaying a certain critical mass phenomenon.\\[5pt]
  On the other hand, if $p \gt 2$,
  we show that for all $\mu$ satisfying $\mu(s) \ge \mu_1 s^{p-2-\eps}$ for all $s \in [0, R]$ and some $\mu_1, \eps \gt 0$
  the system \eqref{prob:abstract} admits a global classical solution for each initial datum $0 \le u_0 \in \con0$.\\[5pt]
  \textbf{AMS Classification (2010):} {35B44 (primary); 35B33, 35K65, 92C17 (secondary)}
\end{abstract}

\section{Introduction}
\newcommand{\growthrate}{\kappa}
We live in a heterogeneous environment
and the fact that for instance growth or death rates may depend on spatial features 
has been incorporated into several models describing population dynamics.
Among the more famous examples is the system
\begin{align} \label{prob:competition}
  \begin{cases}
    u_t = d_1 \Delta u + u [\growthrate(x) - u - v], \\
    v_t = d_2 \Delta v + v [\growthrate(x) - u - v]
  \end{cases}
\end{align}
with $d_1, d_2 \gt 0$ and $\growthrate: \Omega \ra [0, \infty)$,
$\Omega \subset \R$, $n \in \N$, being a smooth, bounded domain,
modelling two species $u$ and $v$ competing for a common resource,
where $\growthrate$ represents a reproduction rate influenced by the environment.

It has the remarkable property that whenever $d_1 \lt d_2$,
then there exists $u_\infty(\growthrate) \gt 0$
such that for any initial data $u_0, v_0 \in \con0$ with $u_0, v_0 \ge 0$ and $v_0 \not\equiv 0$
the corresponding solution $(u, v)$ converges to $(u_\infty(\growthrate), 0)$ --
provided $\growthrate$ is not constant, which reflects spatial heterogeneity (\cite{DockeryEtAlEvolutionSlowDispersal1998}).
If, however, $\growthrate$ is constant 
then $(\lambda \growthrate, (1-\lambda) \growthrate))$ is a steady state of \eqref{prob:competition} for all $\lambda \in [0, 1]$
implying that species with different diffusion rates may coexist in homogeneous environments.
Furthermore, there is considerable activtiy in the analysis of systems similar to \eqref{prob:competition};
for instance, convections terms have been added to these equations (\cite{LouEtAlGlobalDynamicsLotka2018})
and the case of weak competition (\cite{HeNiGlobalDynamicsLotka2016, LouEffectsMigrationSpatial2006})
has been studied in great detail as well.

These results (among others) may arouse interest to consider environmental depending functions in other models as well:
The system
\begin{align} \label{prob:log_pp}
  \begin{cases}
    u_t = \Delta u - \nabla \cdot (u \nabla v) + \kappa u - \mu u^p, \\
    v_t = \Delta v - v + u,
  \end{cases}
\end{align}
in $\Omega \times (0, T)$, where $\Omega \subset \R^n$, $n \in \N$, is a smooth, bounded domain, $T \in (0, \infty]$
and $\kappa, \mu \gt 0$ and $p \ge 1$ are given parameters,
is relevant in the modeling of, for instance, micro- and macroscopic population dynamics
(\cite{HillenPainterUserGuidePDE2009}, \cite{ShigesadaEtAlSpatialSegregationInteracting1979})
or tumor invasion processes
(\cite{ChaplainLolasMathematicalModellingCancer2005}).

For these so-called chemotaxis systems,
at first introduced by Keller and Segel (\cite{KellerSegelTravelingBandsChemotactic1971})
even questions of global existence and boundedness are of great interest.
After all, if one chooses $\kappa = \mu \equiv 0$ in \eqref{prob:log_pp}
in space-dimensions two (\cite{HorstmannWangBlowupChemotaxisModel2001, SenbaSuzukiParabolicSystemChemotaxis2001})
and higher (\cite{WinklerFinitetimeBlowupHigherdimensional2013}) there are initial data leading to blow-up.
For a more broad introduction to Keller--Segel models,
which have been intensively studied in the past decades,
we refer to the survey \cite{BellomoEtAlMathematicalTheoryKeller2015}.

Intuitively, the superlinear degrading term $\mu u^p$ (with $\mu \gt 0$ and $p \gt 1$) in \eqref{prob:log_pp})
should somewhat decrease the possibility of (finite-time) blow-up.
However, exactly how large $\mu$ and $p$ need to be in order to guarantee global existence seems to be an open question,
even for constant $\kappa, \mu \ge 0$.

If $n = 2$ and $\mu \gt 0$
all classical solutions to \eqref{prob:log_pp} exist globally in time (\cite{OsakiEtAlExponentialAttractorChemotaxisgrowth2002}).
One may even replace $u^2$
by a function growing slightly slower than $s \mapsto s^2$ (\cite{XiangSublogisticSourceCan2018}).
The same holds true in higher dimensions,
provided $p \gt 2$ or $p = 2$ and $\mu \gt \frac{n}{4}$ (\cite{WinklerBoundednessHigherDimensionalParabolicParabolic2010}),
while for $p = 2$ and any $\mu \gt 0$ at least global weak solutions have been constructed,
which become smooth after finite time provided $\kappa$ is small enough (\cite{LankeitEventualSmoothnessAsymptotics2015}).

As chemicals can be assumed to diffuse much faster than cells
a typical simplification of \eqref{prob:log_pp} is the parabolic-elliptic system
\begin{align} \label{prob:log_pe}
  \begin{cases}
    u_t = \Delta u - \nabla \cdot (u \nabla v) + \kappa u - \mu u^p, \\
    0   = \Delta v - v + u.
  \end{cases}
\end{align}
For $n = 2$ the conditions $p \ge 2$ and $\mu \gt 0$ suffice to ensure global existence 
while for $n \ge 3$, $p = 2$ and $\mu \ge \frac{n-2}{n}$ or $n \ge 3$, $p \gt 2$ and arbitrary $\mu \gt 0$ the same can be achieved
(\cite{KangStevensBlowupGlobalSolutions2016, TelloWinklerChemotaxisSystemLogistic2007}).

On the other hand,
any thresholds may be surpassed,
if $p = 2$, $\mu \in (0, 1)$
and the diffusion is sufficiently weak, that is,
$\Delta u$ in the first equation in \eqref{prob:log_pp} is replaced by $\eps \Delta u$ for suitable $\eps \gt 0$
(\cite{WinklerHowFarCan2014, LankeitChemotaxisCanPrevent2015}).
This stays in contrast to the case without cross-diffusion as then $\ol u \defs \max\{\|u_0\|_{\leb \infty}, \frac{\kappa}{\mu}\}$
always forms a supersolution
and thus indicates that in chemotaxis systems with logistic source nontrivial structures may emerge
at least on intermediate time scales.

Even more drastic formations are known to form if $p$ is chosen close to (but sill larger than) $1$.
After initial data causing finite-time blow-up have been constructed in dimensions five and higher for certain $p \gt \frac32$
in a system closely related to \eqref{prob:log_pe}
in \cite{WinklerBlowupHigherdimensionalChemotaxis2011},
in \cite{WinklerFinitetimeBlowupLowdimensional2018} finite-time blow-up has also been shown to occur in \eqref{prob:log_pe}
for any $n \ge 3$ and
\begin{align*}
  \begin{cases}
    p \lt \frac76, & n \in \{3, 4\}, \\
    p \lt 1 + \frac{1}{2(n-1)}, & n \ge 5.
  \end{cases}
\end{align*}
Hence, at least in space-dimensions three and higher even superlinear degegration terms do not always ensure global existence.

The case of $\mu$ and $\kappa$ depending on space (and time) has also been studied.
In their three-paper series
\cite{SalakoShenParabolicellipticChemotaxisModel2018,
SalakoShenParabolicellipticChemotaxisModel2018a,
SalakoShenParabolicellipticChemotaxisModel2018b}
Salako and Shen showed inter alia global existence of solutions to \eqref{prob:log_pe} with $\Omega = \R$
provided $\inf_{x \in \Omega} \mu(x) \gt 1$.

\paragraph{Main results}
Apparently, rigorously proving blow-up in Keller--Segel systems is a difficult problem.
Known proofs for parabolic-parabolic chemotaxis systems strongly rely on certain energy structures 
(\cite{CieslakLaurencotFiniteTimeBlowup2010, HorstmannWangBlowupChemotaxisModel2001, WinklerAggregationVsGlobal2010})
while in the parabolic-elliptic setting additional approaches are moment-type arguments
(\cite{BilerLocalGlobalSolvability1998, NagaiBlowupNonradialSolutions2001})

However, all these methods appear inadequate for chemotaxis systems with logistic source.
In this paper we further simplify \eqref{prob:log_pe} and consider
\begin{align} \label{prob:p} \tag{P}
  \begin{cases}
    u_t = \Delta u - \nabla \cdot (u \nabla v) + \kappa(|x|) u - \mu(|x|) u^p,        & \text{in $\Omega \times (0, T)$}, \\
    0   = \Delta v - \frac{m(t)}{|\Omega|} + u, \quad m(t) \defs \intom u(\cdot, t),  & \text{in $\Omega \times (0, T)$}, \\
    \partial_\nu u = \partial_\nu v = 0,                                              & \text{on $\partial \Omega \times (0, T)$}, \\
    u(\cdot, 0) = u_0,                                                                & \text{in $\Omega$}
  \end{cases}
\end{align}
for given functions $\kappa, \mu, u_0: \Omega \ra \R$ and $T \in (0, \infty]$
where we henceforth fix $R \gt 0$ and $\Omega \defs B_R(0) \subset \R^2$.
Our main results are the following.

\begin{theorem} \label{th:blow_up}
  Let $p \ge 1$, $\alpha \ge 2(p - 1)$, $\mu_1 \gt 0$ and
  suppose that $\kappa, \mu \in C^0([0, R]) \cap C^1((0, R))$ satisfy
  \begin{align} \label{eq:blow_up:cond_kappa_mu}
    \kappa, -\kappa', \mu, \mu' \ge 0 \quad \text{in $(0, R)$}
  \end{align}
  as well as
  \begin{align} \label{eq:blow_up:cond_mu}
    \mu(s) \le \mu_1 s^\alpha \quad \text{for all $s \in [0, R]$}.
  \end{align}
  For any $m_0 \gt 8\pi$ there exist $r_1 \in (0, R)$ and $\tilde m \in (0, m_0)$ such that if
  \begin{align} \label{eq:blow_up:cond_u0}
    0 \le u_0 \in C^0(\ombar) \quad \text{is radially symmetric and radially decreasing} 
  \end{align}
  with
  \begin{align} \label{eq:blow_up:mass_concentration}
    \intom u_0 = m_0
    \quad \text{and} \quad
    \int_{B_{r_1}(0)} u_0 \ge \tilde m,
  \end{align}
  then there exists a classical solution $(u, v)$ to \eqref{prob:p} with initial datum $u_0$ blowing up in finite time;
  that is, there exists $\tmax \in (0, \infty)$ such that
  \begin{align} \label{eq:blow_up:limsup_u}
    \limsup_{t \nea \tmax} \|u(\cdot, t)\|_{L^\infty(\Omega)} = \infty.
  \end{align}
\end{theorem}

\begin{remark}
  To give a more concrete example, 
  the conditions \eqref{eq:blow_up:cond_kappa_mu} and \eqref{eq:blow_up:cond_mu} are for instance fulfilled
  if $p = 2$, $\kappa \ge 0$ is a constant and $\mu(r) = r^2, r \in [0, R]$.
\end{remark}

This result will be complemented by two statements on global solvability.
Firstly, we show at least in the case $\kappa \equiv 0$ the value $8\pi$
-- which does not, as one could have expected, depend on $\alpha$ or $p$
-- is essentially optimal.
\begin{prop} \label{prop:critical_mass}
  Let $\kappa \equiv 0$, $0 \le \mu \in C^0([0, R]) \cap C^1((0, R))$ and $p \ge 1$.
  For any nonnegative radially symmetric $u_0 \in C^0(\ombar)$ with $\intom u_0 \lt 8\pi$ there exists a global classical solution
  $(u, v)$ to \eqref{prob:p} with initial datum $u_0$.
\end{prop}

Secondly, if $p \gt 2$, we prove that for arbitrary initial data global classical solutions exist
provided $\mu$ does not grow too fast.
\begin{prop} \label{prop:global_ex}
  Let $p \gt 2$, $\alpha \lt p - 2$, $\mu_1 \gt 0$ and $\kappa, \mu \in C^0([0, R]) \cap C^1((0, R))$.
  If
  \begin{align} \label{eq:global_ex:cond_mu}
    \mu(s) \ge \mu_1 s^\alpha \quad \text{for all $s \in [0, R]$}
  \end{align}
  then \eqref{prob:p} admits a global classical solution for any nonnegative initial datum $u_0 \in \con0$.
\end{prop}

\paragraph{Plan of the paper}
For the proof of Theorem~\ref{th:blow_up} we will rely on a transformation
introduced by Jäger and Luckhaus in \cite{JagerLuckhausExplosionsSolutionsSystem1992}.
As will be seen in Lemma~\ref{lm:pde_w} below the function $w: [0, R]^2 \times [0, \tmax) \ra \R$ defined by
\begin{align*}
  w(s, t) \defs \int_0^{\sqrt s} \rho u(\rho, t) \diff \rho, \quad s \in [0, R^2], t \in [0, \tmax),
\end{align*}
solves the \emph{scalar} PDI
\begin{align} \label{eq:intro:p_pdi}
          w_t
    &\ge  4s w_{ss}
          + 2 w w_s
          - \frac{m(t)}{|\Omega|} s w_s
          - 2^{p-1} \int_0^s \mu(\sqrt \sigma) w_s^p(\sigma, \cdot) \dsigma
    \quad \text{in $(0, R^2) \times (0, \tmax)$}.
\end{align}

In similar -- but higher dimensional -- settings for certain $s_0, \gamma \gt 0$ the function
\begin{align*}
  \phi: [0, \tmax) \ra \R; \quad t \mapsto \int_0^{s_0} s^{-\gamma} (s-s_0) w(s, t) \ds,
\end{align*}
where $w$ denotes a similar transformed quantity,
has been shown to solve a certain ODI implying finite-time blow-up
(\cite{WinklerCriticalBlowupExponent2018}, \cite{WinklerFinitetimeBlowupLowdimensional2018}).

However, these techniques seem to be insufficient to provide any insights in the two dimensional setting,
as the term stemming from the diffusion can apparently not be dealt with anymore.

Therefore, we follow a different approach.
In order to show finite-time blow-up for \eqref{prob:p} with $\kappa = \mu \equiv 0$ in the planar setting
Winkler (\cite{WinklerHowUnstableSpatial2018}) has recently utilized the function
\begin{align*}
  \phi: [0, \tmax) \ra \R; \quad t \mapsto \int_0^{s_0} (s-s_0)^\beta w(s, t) \ds
\end{align*}
for certain $s_0, \beta \gt 0$ instead.
Most terms in \eqref{eq:intro:p_pdi} can be dealt similarly as in \cite{WinklerHowUnstableSpatial2018}
-- except for the nonlocal term $\int_0^s \mu(\sqrt \sigma) w_s^p(\sigma, \cdot) \dsigma$ which is, of course,
not present if $\mu \equiv 0$.

The main idea for dealing with this integral is to derive a pointwise bound for $w_s$ (Lemma~\ref{lm:ws_bdd})
and then integrate by parts,
where the condition $\alpha \ge 2(p-1)$ is apparently needed in order to able to handle the remaining terms (Lemma~\ref{lm:i4}).

Finally, we will then see by an ODI comparison argument
that for suitably chosen initial data $\phi$ (and hence $u$) cannot exist globally in time.

\section{Preliminaries}
The following statement on local existence, in its essence based on a fixed point argument, is standard.
Hence we may omit a proof here and just refer to, for instance,
\cite{CieslakWinklerFinitetimeBlowupQuasilinear2008} or \cite{TelloWinklerChemotaxisSystemLogistic2007}
for more detailed arguments in similar frameworks.
\begin{lemma} \label{lm:local_ex}
  Let $0 \le u_0 \in C^0(\ombar)$ and $\kappa, \mu \in C^0([0, R]) \cap C^1((0, R))$.
  Then there exist $\tmax \in (0, \infty]$ and a classical solution $(u, v)$ to \eqref{prob:p}
  uniquely determined by
  \begin{align*}
    u &\in C^0(\ombar \times [0, \tmax)) \cap C^{2, 1}(\ombar \times (0, \tmax)), \\
    v &\in \bigcap_{q \gt 2} C^0([0, \tmax); W^{1, q}(\Omega)) \cap C^{2, 0}(\ombar \times (0, \tmax))
  \end{align*}
  and
  \begin{align*}
    \intom v(\cdot, t) = 0 \quad \text{for all $t \in (0, \tmax)$}.
  \end{align*}
  Moreover, this solution is nonnegative in the first component, radially symmetric if $u_0$ is radially symmetric
  and such that if $\tmax \lt \infty$ then
  \begin{align*}
    \limsup_{t \nea \tmax} \|u(\cdot, t)\|_{L^\infty(\Omega)} = \infty.
  \end{align*}
\end{lemma}

Unless otherwise stated we henceforth fix $u_0 \in \con0$ satisfying \eqref{eq:blow_up:cond_u0}
as well as $\kappa, \mu \in C^0([0, R]) \cap C^1((0, R))$ fulfilling \eqref{eq:blow_up:cond_kappa_mu}
and denote the corresponding solution provided by Lemma~\ref{lm:local_ex} by $(u, v)$ as well as the maximal existence time by $\tmax$.
Finally, we set $m_0 \defs m(0)$ and $\kappa_1 \defs \|\kappa\|_{L^\infty((0, R))}$.

\begin{lemma} \label{lm:mass}
  For all $t \in (0, \tmax)$ the inequalities
  \begin{align*}
    0 \le m(t) \le m_0 \ure^{\kappa_1 t}
  \end{align*}
  hold.
\end{lemma}
\begin{proof}
  Nonnegativity of $u$ implies $m \ge 0$
  while an ODI comparison argument yields $m(t) \le m_0 \ure^{\kappa_1 t}$ for $t \gt 0$
  due to $m' \le \kappa_1 m$ in $(0, \tmax)$.
\end{proof}

As mentioned in the introduction the proof of Theorem~\ref{th:blow_up} will rely on transforming \eqref{prob:p} into a scalar equation.
\begin{lemma} \label{lm:pde_w}
  Define
  \begin{align*}
    w(s, t) \defs \int_0^{\sqrt{s}} \rho u(\rho, t) \drho, \quad s \in [0, R^2], t \in [0, \tmax).
  \end{align*}
  Then
  \begin{align} \label{eq:pde_w:w_s_eq_u}
    w_s(s, t) = \tfrac12 u(\sqrt s, t)
  \end{align}
  and
  \begin{align} \label{eq:pde_w:pde}
          w_t(s, t)
    &=    4s w_{ss}(s, t)
          + 2 w(s, t) w_s(s, t)
          - \frac{m(t)}{|\Omega|} s w_s(s, t) \notag \\
    &\pe  + \int_0^s \left(\kappa(\sqrt \sigma) w_s(\sigma, t) - 2^{p-1} \mu(\sqrt \sigma) w_s^p(\sigma, t) \right) \dsigma
  \end{align}
  for $s \in (0, R^2)$ and $t \in (0, \tmax)$.
\end{lemma}
\begin{proof}
  The first two equations in \eqref{prob:p} read in radial form
  \begin{align*}
    u_t &= \frac1r (r u_r - r u v_r)_r + \kappa(r) u - \mu(r) u^p \quad \text{and} \\ 
    0   &= \frac1r (r v_r)_r -\frac{m(t)}{|\Omega|} + u, 
  \end{align*}
  that is
  \begin{align*}
      r v_r(r, \cdot)
    = \int_0^r \left( \frac{m(t)}{|\Omega|} \rho - \rho u(\rho, \cdot) \right) \drho
    = \frac{m(t)}{2|\Omega|} r^2 - w(r^2, \cdot).
  \end{align*}
  Thus, a direct calculation yields 
  \begin{align*}
    w_s(s, t)     &=  \frac{1}{2\sqrt s} \cdot \sqrt s u(\sqrt s, t) = \frac12 u(\sqrt s, t), \\
    w_{ss}(s, t)  &=  \frac12 u_r(\sqrt s, t) \cdot \frac1{2 \sqrt s} = \frac1{4 \sqrt s} u_r(\sqrt s, t) \quad \text{and} \\
    w_t(s, t)     &=  \int_0^{\sqrt s} \frac{\rho}{\rho} [\rho u_r(\rho, t) - \rho u(\rho, t) v_r(\rho, t)]_r \drho
                      + \int_0^{\sqrt s} \rho [\kappa(\rho) u(\rho, t) - \mu(\rho) u^p(\rho, t) ] \drho \\
                  &=  \sqrt s u_r(\sqrt s,  t) - u(\sqrt s, t) \left[ \frac{m(t)}{2|\Omega|} s - w(s, t) \right]
                      - \frac12 \int_0^s \left( \kappa(\sqrt \sigma) u + \mu(\sqrt \sigma) u^p(\sqrt \sigma, t) \right) \dsigma \\
                  &=  4 s w_{ss}(s, t) + 2 w(s, t) w_s(s, t) - \frac{m(t)}{|\Omega|} s w_s(s, t)
                      - \int_0^s \left(\kappa(\sqrt \sigma) w_s + 2^{p-1} \mu(\sqrt \sigma) w_s^p(\sigma, t) \right) \dsigma
  \end{align*}
  for $s \in (0, R^2)$ and $t \in (0, \tmax)$.
\end{proof}

\section{Supercritical mass allows for blow-up}
Crucially relying on transforming \eqref{prob:p} into the scalar equation \eqref{eq:pde_w:pde}
we will prove Theorem~\ref{th:blow_up} at the end of this section.

\subsection{The function $\phi$}
\begin{lemma} \label{lm:phi_first_ode}
  Let $\beta \gt -1$ and $s_0 \in (0, R^2)$.
  The function
  \begin{align*}
    \phi: [0, \tmax) \ra \R, \quad
    t \mapsto \intns (s_0-s)^\beta w(s, t) \ds
  \end{align*}
  belongs to $C^0([0, \tmax)) \cap C^1((0, \tmax))$ and satisfies
  \begin{align} \label{eq:phi_first_ode:ode}
          \phi'(t)
    &\ge  4 \intns (s_0-s)^\beta s w_{ss}(s, t) \ds \notag \\
    &\pe  + 2 \intns (s_0-s)^\beta s w(s, t) w_s(s, t) \ds \notag \\
    &\pe  - \frac{m(t)}{|\Omega|} \intns (s_0-s)^\beta s w_s(s, t) \ds \notag \\
    &\pe  - 2^{p-1} \intns \int_0^s (s_0-s)^\beta \mu(\sqrt{\sigma}) w_s^p(\sigma, t) \dsigma \ds \notag \\
    &\sfed I_1(t) + I_2(t) + I_3(t) + I_4(t)
  \end{align}
  for all $t \in (0, \tmax)$.
\end{lemma}
\begin{proof}
  As $w \in C^0(\ombar \times [0, \tmax)) \cap C^1(\ombar \times (0, \tmax))$ by Lemma~\ref{lm:pde_w},
  the asserted regularity of $\phi$ follows from standard Lebesgue integration theory,
  while \eqref{eq:phi_first_ode:ode} is then a direct consequence of Lemma~\ref{lm:pde_w} and nonnegativity of $u$ and $\kappa$.
\end{proof}

Our goal is to show that after an appropriate choice of parameters $\phi$ satisfies a certain ODI,
which then implies finiteness of $\tmax$.
\begin{lemma} \label{lm:ode_blow_up}
  Let $T, \tilde T, c_1, c_2, c_3 \gt 0$.
  If $y \in C^0([0, T)) \cap C^1((0, T))$ satisfies
  \begin{align*}
    \begin{cases}
      y' \ge c_1 y^2 - c_2 y - c_3, \\
      y(0) \ge y_0
    \end{cases}
  \end{align*}
  in $(0, T)$ with
  \begin{align*}
    y_0 \ge \frac{c_2 + \sqrt{c_1 c_3}}{c_1} + \frac1{c_1 \tilde T},
  \end{align*}
  then necessarily $T \le \tilde T$.
\end{lemma}
\begin{proof}
  As
  \begin{align*}
    c_1 s^2 - c_2 s - c_3 = 0
    \quad \text{if and only if} \quad
      s
    = \frac{c_2 \pm \sqrt{c_2^2 + 4 c_1 c_3}}{2 c_1}
    \sfed \lambda_\pm
  \end{align*}
  the ODI implies that $y$ is increasing if and only if $y \le \lambda_-$ or $y \ge \lambda_+$.
  Since
  \begin{align*}
    \lambda_+ \le \frac{c_2 + \sqrt{c_1 c_3}}{c_1} \lt y_0
  \end{align*}
  and
  \begin{align*}
    (s-\lambda_-)(s-\lambda_+) \ge (s-\lambda_+)^2
    \quad \text{for all $s \ge \lambda_+$}
  \end{align*}
  we conclude that $y$ is indeed increasing in $(0, T)$ and satisfies
  \begin{align*}
    y' \ge c_1 (y - \lambda_+)^2
  \end{align*}
  in $(0, T)$.

  Hence by integrating we obtain
  \begin{align*}
        t
    =   \int_0^t 1 \ds 
    \le \int_{y(0)}^{y(t)} \frac1{c_1 (y - \lambda_+)^2}
    \le \frac1{c_1 (y_0 - \lambda_+)} - \frac1{c_1 (y(t) - \lambda_+)}
    \lt \tilde T - 0
    =   \tilde T
    \quad \text{for all $t \in (0, T)$},
  \end{align*}
  which is absurd for $T \gt \tilde T$.
\end{proof}

Apart from the nonlocal term in \eqref{lm:phi_first_ode} all integrals therein as well as $\phi(0)$
can be estimated as in \cite[Lemma~3.2]{WinklerHowUnstableSpatial2018}.
For sake of completeness we nonetheless give short proofs for the following lemmata.
\begin{lemma} \label{lm:phi_0}
  Let $\beta \gt -1$ and $s_0 \in (0, R^2)$
  as well as $\tilde m \in (0, m)$ and $\lambda \in (0, 1)$.
  If
  \begin{align*}
    \int_{B_{r_1}(0)} u_0 \ge \tilde m
  \end{align*}
  with $r_1 \defs (\lambda s_0)^2$,
  then
  \begin{align*}
    \phi(0) \ge \frac{\tilde m}{2 \pi (\beta+1)} ((1-\lambda) s_0)^{\beta+1}.
  \end{align*}
\end{lemma}
\begin{proof}
  Set $s_1 \defs \lambda s_0$. As $w_0$ is increasing (due to $u_0 \ge 0$) we have
  \begin{align*}
          \phi(0)
    &=    \intns (s_0-s)^\beta w_0(s) \ds \\
    &\ge  \int_{s_1}^{s_0} (s_0-s)^\beta w_0(s_1) \ds \\
    &=    \int_0^{\sqrt{s_1}} \rho u_0(\rho) \drho \int_{s_1}^{s_0} (s_0-s)^\beta \ds  \\
    &\ge  \frac{\tilde m}{2 \pi} \cdot \frac{((1-\lambda) s_0)^{\beta+1}}{\beta+1}.
    \qedhere
  \end{align*}
\end{proof}

\begin{lemma} \label{lm:i1}
  Let $\beta \gt 1$ and $s_0 \in (0, R^2)$.
  Then for all $t \in (0, \tmax)$
  \begin{align} \label{eq:i1:statement}
    I_1(t) \ge -\frac{2}{\pi} s_0^\beta m_0 \ure^{\kappa_1 t}
  \end{align}
  holds, where $I_1$ is defined in \eqref{eq:phi_first_ode:ode}.
\end{lemma}
\begin{proof}
  By integrating by parts twice
  we obtain for $t \in (0, \tmax)$
  \begin{align*}
          I_1(t)
    &=    4 \intns (s_0-s)^\beta s w_{ss}(s, t) \ds \\
    &=    4 \intns \left( \beta (s_0-s)^{\beta-1} s - (s_0-s)^\beta \right) w_s(s, t) \ds + 0 \\
    &=    4 \intns (s_0-s)^{\beta-1} \left( (\beta+1) s - s_0 \right)  w_s(s, t) \ds \\
    &=    4 \intns (s_0-s)^{\beta-2} \left[ (\beta-1) \left( (\beta+1) s - s_0 \right) - (\beta+1) (s_0-s) \right] w(s, t) \ds \\
    &=    - 8 \beta \intns (s_0-s)^{\beta-2} \left(s_0 - \frac{\beta+1}{2} s\right) w(s, t) \ds.
  \intertext{As $w(\cdot, t)$ is nonnegative and increasing by \eqref{eq:pde_w:w_s_eq_u} and Lemma~\ref{lm:mass},
  noting that $s_0 - \frac{\beta+1}{2} s \le 0$ if and only if $s \ge s_1 \defs \frac{2s_0}{\beta+1} \in (0, s_0)$,
  we conclude
  }
          I_1(t)
    &\ge  - 8 \beta \intns (s_0-s)^{\beta-2} \left(s_0 - \frac{\beta+1}{2} s\right) w(s_1, t) \ds \\
    &=    - 4 s_0^\beta w(s_1, t).
  \end{align*}
  Because the definition of $w$ and Lemma~\eqref{lm:mass} warrant that
  \begin{align*}
        w(s_1, t)
    \le w(R^2, t)
    =   \frac{m(t)}{2 \pi}
    \le \frac{m_0 \ure^{\kappa_1 t}}{2 \pi}
    \quad \text{for $t \in (0, \tmax)$}
  \end{align*}
  a consequence thereof is \eqref{eq:i1:statement}.
\end{proof}

\begin{lemma} \label{lm:i2_i3}
  Let $\beta \gt 0$, $s_0 \in (0, R^2)$ and $\eta \in (0, 1)$.
  With $I_2$ and $I_3$ as in \eqref{eq:phi_first_ode:ode}
  \begin{align} \label{eq:i2_i3:statement}
        I_2(t) + I_3(t)
    \ge (1-\eta) \frac{\beta(\beta+2)}{s_0^{\beta+2}} \phi^2(t)
        - \frac{m_0^2 \ure^{2\kappa_1 t}}{2 \eta (\beta+1)(\beta+2) |\Omega|^2 } s_0^{\beta+2}
  \end{align}
  holds then for all $t \in (0, \tmax)$.
\end{lemma}
\begin{proof}
  Let $t \in (0, \tmax)$.
  An integration by parts yields
  \begin{align*}
        I_2(t)
    &=  2\intns (s_0-s)^\beta w(s, t) w_s(s, t) \ds \\
    &=  \intns (s_0-s)^\beta (w^2)_s(s, t) \ds \\
    &=  \beta \intns (s_0-s)^{\beta-1} w^2(s, t) \ds + \left[ (s_0-s)^\beta w^2(s, t) \right]_0^{s_0} \\
    &=  \beta \intns (s_0-s)^{\beta-1} w^2(s, t) \ds
  \end{align*}
  while by another integration by parts and Young's inequality we have
  \begin{align*}
          I_3(t)
    &=    - \frac{m(t)}{|\Omega|} \intns (s_0-s)^\beta s w_s(s, t) \ds \\
    &=    \frac{m(t)}{|\Omega|} \intns (s_0-s)^\beta w(s, t) \ds
          - \frac{\beta m(t)}{|\Omega|} \intns (s_0-s)^{\beta-1} s w(s, t) \ds
          + 0 \\
    &\ge  0
          - \eta \beta \intns (s_0-s)^{\beta-1} w^2(s, t) \ds
          - \frac{\beta m^2(t)}{4\eta |\Omega|^2} \intns (s_0-s)^{\beta-1} s^2 \ds \\
    &\ge  - \eta \beta \intns (s_0-s)^{\beta-1} w^2(s, t) \ds
          - \frac{m_0^2 \ure^{2\kappa_1 t}}{2 \eta (\beta+1)(\beta+2) |\Omega|^2 } s_0^{\beta+2}.
  \end{align*}
  As also by Hölder's inequality
  \begin{align*}
          \phi(t)
    &=    \intns (s_0-s)^\beta w(s, t) \ds
    \le   \left(\intns (s_0-s)^{\beta+1} \ds \right)^\frac12
          \left(\intns (s_0-s)^{\beta-1} w^2(s, t) \ds \right)^\frac12,
  \end{align*}
  that is, 
  \begin{align*}
          \phi^2(t)
    &\le  \frac{s_0^{\beta+2}}{\beta+2} \intns (s_0-s)^\beta w^2(s, t) \ds,
  \end{align*}
  we conclude \eqref{eq:i2_i3:statement}.
\end{proof}

\subsection{The fourth integral}
In order to be able to advantageously integrate by parts in the nonlocal term in \eqref{lm:phi_first_ode}
we first derive a pointwise bound for $w_s$, which in turn is prepared by the following two lemmata.

\begin{lemma} \label{lm:v_rr_le_u}
  In $(0, R) \times (0, \tmax)$ the inequality $-v_{rr} \le u$ holds.
\end{lemma}
\begin{proof}
  As $u \ge 0$ we have by the second equation in \eqref{prob:p}
  \begin{align*}
    (r v_r(r, t))_r \le r \frac{m(t)}{|\Omega|} 
    \quad \text{for $(r, t) \in (0, R) \times (0, \tmax)$},
  \end{align*}
  hence upon integrating
  \begin{equation*}
    v_r(r, t) \le \frac{r}{2} \frac{m(t)}{|\Omega|}
    \quad \text{for $(r, t) \in (0, R) \times (0, \tmax)$}.
  \end{equation*}
  Again by the second equation in \eqref{prob:p} we have $v_{rr} = \frac{m(t)}{|\Omega|} - u - \frac1r v_r$
  such that a direct consequence thereof is $v_{rr} \ge -u$.
\end{proof}

\begin{lemma} \label{lm:ur_le_0}
  Throughout $(0, R) \times (0, \tmax)$ we have $u_r \le 0$.
\end{lemma}
\begin{proof}
  Without loss of generality we may assume that $u_0 \in C^2(\ombar)$ with $\partial_\nu u_0 = 0$ on $\partial \Omega$,
  as for less regular initial data the statement follows by an approximation procedure
  as in \cite[Lemma~2.2]{WinklerCriticalBlowupExponent2018}.

  Since additionally $\sup_{(x, t) \in [0, R] \times [0, T]} |\nabla v(x, t)| \lt \infty$
  by elliptic regularity theory (cf.\ \cite[Theorem~19.1]{FriedmanPartialDifferentialEquations1976}) for all $T \in (0, \tmax)$
  we may invoke \cite[Theorem~1.1]{LiebermanHolderContinuityGradient1987} to obtain
  \begin{align*} 
    u \in C^{1, 0}(\ol \Omega \times [0, \tmax)) \cap C^{3, 1}(\ol \Omega \times (0, \tmax)).
  \end{align*}
  Hence, fixing $T \in (0, \tmax)$ and letting $Q_T \defs [0, R] \times [0, T]$,
  the function $z \defs u_r|_{Q_T}$
  belongs to $C^{0}(Q_T)$ as well as to $C^{2, 1}([0, R] \times (0, T))$
  and satisfies,
  due to $u_t = u_{rr} + \frac1r u_r - u_r v_r - u \left(\frac{m(t)}{|\Omega|} - u\right) + \kappa(r) u - \mu(r) u^p$ in $Q_T$,
  \begin{align*}
    z_t = z_{rr} + a(r, t) z_r + b(r, t) z + c(r, t)
    \quad \text{in $Q_T$},
  \end{align*}
  wherein
  \begin{align*}
    a(r, t) &\defs  \frac1r - v_r(r, t), \\
    b(r, t) &\defs  -\frac1{r^2} - v_{rr}(r, t) - \frac{m(t)}{|\Omega|} + 2u(r, t) + \kappa(r) - p\mu(r) u^{p-1} \quad \text{and} \\
    c(r, t) &\defs  \kappa'(r) u - \mu'(r) u^p(r, t)
  \end{align*}
  for $(r, t) \in Q_T$.
  
  As $\kappa' \le 0$ and $\mu' \ge 0$ by \eqref{eq:blow_up:cond_kappa_mu},
  $u_r(0, \cdot) = 0$ due to radial symmetry,
  $u_r(R, \cdot) \le 0$ since $u \gt 0$ in $(0, R)$ 
  and $u_{0r} \le 0$ because of \eqref{eq:blow_up:cond_u0} we have
  \begin{align*} 
    \begin{cases}
      z_t \le z_{rr} + a(r, t) z_r + b(r, t) z & \text{in $(0, R) \times (0, T)$}, \\
      z \le 0,                                 & \text{on $\{0, R\} \times (0, T)$}, \\
      z(\cdot, 0) \le 0,                       & \text{in $(0, R)$}.
    \end{cases}
  \end{align*}
  Lemma~\ref{lm:v_rr_le_u} warrants that $-v_{rr} \le u$ in $Q_T$,
  hence $\sup_{(r, t) \in Q_T} b(x, t) \le 3u(r, t) + \kappa(r) \lt \infty$,
  such that the comparison principle \cite[Proposition~52.4]{QuittnerSoupletSuperlinearParabolicProblems2007}
  becomes applicable and yields $z \le 0$.
  The statement follows then upon taking $T \nea \tmax$.
\end{proof}

\begin{lemma} \label{lm:ws_bdd}
  We have
  \begin{align*}
    w_s(s, t) \le \frac{m_0 \ure^{\kappa_1 t}}{2 \pi s}
  \end{align*}
  for all $s \in (0, R^2), t \in (0, \tmax)$.
\end{lemma}
\begin{proof}
  Let $r \in (0, R)$ and $t \in (0, \tmax)$.
  On the one hand we have by Lemma~\ref{lm:ur_le_0} 
  \begin{align*}
        \intom u(\cdot, t)
    =   2\pi \int_0^R \rho u(\rho, t) \drho
    \ge 2\pi \int_0^r \rho u(r, t) \drho
    =   \pi r^2 u(r, t)
  \end{align*}
  and one the other hand by Lemma~\ref{lm:mass}
  \begin{align*}
        \intom u(\cdot, t)
    \le m_0 \ure^{\kappa_1 t}
  \end{align*}
  such that
  \begin{align*}
        u(r, t) \le \frac{m_0 \ure^{\kappa_1 t}}{\pi r^2}.
  \end{align*}
  
  The statement follows due to $w_s(s, t) = \frac12 u(s^\frac12, t)$ for $s \in (0, R^2)$ and $t \in (0, \tmax)$.
\end{proof}

\begin{remark}
  The exponent $-1$ in Lemma~\ref{lm:ws_bdd} is essentially optimal.
  Indeed, if we were able to show $w_s(s, t) \le f(t) s^{-q}$ for some $f \in C^0([0, \infty))$ and $q \lt 1$
  and all $(s, t) \in (0, R^2) \times (0, \tmax)$,
  then also $u(r, t) \le 2 f(t) r^{-2q}$ for $r \in (0, R)$ and $t \in (0, \tmax)$.
  However, this would yield $\sup_{t \in (0, T)} \|u(\cdot, t)\|_{\leb\lambda} \lt \infty$
  for some $\lambda \gt 1$ and all finite $T \in (0, \tmax]$,
  which in turn would rapidly imply $\tmax = \infty$, confer the proof of Proposition~\ref{prop:global_ex} below.
\end{remark}

With these preparations at hand
we are finally able to deal with the fourth integral on the right-hand side of \eqref{lm:phi_first_ode}.
\begin{lemma} \label{lm:i4}
  Let $\beta \gt -1$, $s_0 \in (0, \min\{1, R^2\})$
  and suppose that $\mu$ satisfies \eqref{eq:blow_up:cond_mu} for some $\mu_1 \gt 0$ and $\alpha \ge 2(p-1)$.
  Then
  \begin{align*}
        2^{p-1} \intns \int_0^s (s_0-s)^\beta \mu(\sqrt{\sigma}) w_s^p(\sigma, t) \dsigma \ds
    \le C \phi(t)
  \end{align*}
  for all $t \in (0, \tmmax)$, where $C \defs \left(\frac{m_0 \ure^\kappa}{\pi}\right)^{p-1} \mu_1$
  and $\tmmax \defs \min\{1, \tmax\}$.
\end{lemma}
\begin{proof}
  Let $\alpha ' := \frac{\alpha}{2} - (p-1)$.
  Due to \eqref{eq:blow_up:cond_mu} we see that $\alpha' \ge 0$,
  such that an application of Lemma~\ref{lm:ws_bdd} and an integration by parts yield 
  \begin{align*}
    &\pe  2^{p-1} \intns (s_0-s)^\beta \int_0^s \mu(\sqrt{\sigma}) w_s^p(\sigma, t) \dsigma \ds \\
    &\le  \left(\frac{2m_0 \ure^\kappa}{2\pi}\right)^{p-1} \mu_1
            \intns (s_0-s)^\beta \int_0^s \sigma^{\alpha'} w_s(\sigma, t) \dsigma \ds \\
    &=    C  \intns (s_0-s)^\beta \left[
            - \alpha' \int_0^s \sigma^{\alpha'-1} w(\sigma, t) \dsigma
            + s^{\alpha'} w(s, t)
          \right] \ds \\
    &\le  C \intns (s_0-s)^\beta w(s, t)
    =     C \phi(t)
  \end{align*}
  for $t \in (0, \tmmax)$.
\end{proof}

\subsection{Conclusion. Proof of Theorem~\ref{th:blow_up}}
As it turns out, for any initial mass $m_0 \gt 8\pi$ we are able to find a suitable initial datum $u_0$ with $\intom u_0 = m_0$
as well as sufficiently small $s_0$ and sufficiently large $\beta$
such that a combination of the estimates above makes Lemma~\ref{lm:ode_blow_up} applicable --
implying that $\phi$ and hence $u$ must blow up in finite time.

\begin{proof}[Proof of Theorem~\ref{th:blow_up}]
  Let $m_0 \gt 8\pi$ and $\mu_1 \gt 0$.

  The function   
  \begin{equation*}
    f: (0, m_0] \times [0, \tmax) \times (1, \infty) \times [0, 1) \times [0, 1) \ra \R
  \end{equation*}
  defined by
  \begin{equation*}
    (\tilde m, \tilde T, \beta, \lambda, \eta) \mapsto
    (1-\eta) \beta (\beta+2)
    \cdot \frac{\tilde m^2}{4\pi^2(\beta+1)^2} (1-\lambda)^{2\beta + 2}
    \cdot \frac{\pi}{2 m_0 \ure^{\kappa \tilde T}}
  \end{equation*}
  is continuous and satisfies
  \begin{align*} 
      \lim_{\beta \nea \infty}
      f(m_0, 0, \beta, 0, 0)
    = \frac{m_0}{8\pi}.
  \end{align*}

  Thus, due to our assumption that $m_0 \gt 8\pi$ we may first choose $\beta \in (1, \infty)$
  and then $\tilde m \in (0, m)$, $\tilde T \in (0, \min\{1, \tmax\})$, $\lambda \in (0, 1)$ and $\eta \in (0, 1)$
  as well as $\eps \in (0, 1)$ such that
  \begin{equation} \label{eq:blow_up:f_ge_1}
    (1-\eps)^2 f(\tilde m, \tilde T, \beta, \lambda, \eta) \ge 1.
  \end{equation}

  For $s_0 \gt 0$ let
  \begin{align*}
    c_1(s_0)      &\defs (1-\eta) \frac{\beta(\beta+2)}{s_0^{\beta+2}}, \\
    c_2(s_0)      &\defs \left(m_0 \ure^\kappa \pi^{-1}\right)^{p-1} \mu_1, \\
    c_{3,1}(s_0)  &\defs \frac{m_0^2 \ure^{2\kappa_1 t}}{2 \eta (\beta+1)(\beta+2) |\Omega|^2 } s_0^{\beta+2} \quad \text{and} \\
    \phi_0(s_0)   &\defs \frac{\tilde m}{2 \pi (\beta+1)} ((1-\lambda) s_0)^{\beta+1},
  \end{align*}
  then there exist $d_1, d_2, d_3 \gt 0$ such that
  \begin{align*}
    \frac{c_2(s_0)}{c_1(s_0) \phi_0(s_0)}       = d_1 s_0, \quad
    \frac{c_{3,1}(s_0)}{c_1(s_0) \phi_0^2(s_0)} = d_2 s_0^2
    \quad \text{and} \quad
    \frac{1}{c_1(s_0) \phi_0(s_0)}              = d_3 s_0
  \end{align*}
  for all $s_0 \gt 0$.

  Hence we may choose $s_0 \in (0, \min\{1, R^2\})$ small enough such that
  \begin{align*}
    \frac{\eps}{3} \phi_0(s_0)                   \ge \frac{c_2(s_0)}{c_1(s_0)}, \quad 
    \left( \frac{\eps}{3}  \phi_0(s_0) \right)^2 \ge \frac{c_{3, 2}(s_0)}{c_1(s_0)} \quad \text{and} \quad
    \frac{\eps}{3} \phi_0(s_0)                   \ge \frac{2}{c_1(s_0) \tilde T}.
  \end{align*}
  Set also
  \begin{align*}
    c_{3,2}(s_0)  &\defs \frac{2 s_0^\beta m_0 \ure^{\kappa \tilde T}}{\pi},
  \end{align*}
  then
  \begin{align*}
    \left( (1-\eps) \phi_0(s_0) \right)^2 \ge \frac{c_{3, 2}(s_0)}{c_1(s_0)}
  \end{align*}
  by \eqref{eq:blow_up:f_ge_1}.

  Suppose now that $\kappa, \mu \in C^1([0, R])$ comply with \eqref{eq:blow_up:cond_kappa_mu} and \eqref{eq:blow_up:cond_mu}
  and that $u_0$ satisfies \eqref{eq:blow_up:cond_u0} and \eqref{eq:blow_up:mass_concentration} with $r_1 \defs (\lambda s_0)^2$,
  but that the corresponding solution $(u, v)$ given by Lemma~\ref{lm:local_ex} is global in time.
  Due to the lemmata above the function $\phi$ defined in Lemma~\ref{lm:phi_first_ode} would then fulfill 
  \begin{align*}
    \begin{cases}
      \phi'(t)  \ge c_1 \phi^2(t) - c_2 \phi(t) - c_{3,1} - c_{3, 2}, \quad t \in (0, \tilde T), \\
      \phi(0) \ge \phi_0,
    \end{cases}
  \end{align*}
  where we abbreviated $c_i \defs c_i(s_0)$ and $\phi_0 \defs \phi_0(s_0)$.

  However, as
  \begin{align*}
          \phi_0
    &=    \frac{\eps}{3} \phi_0
          + \frac{\eps}{3} \phi_0
          + (1-\eps) \phi_0
          + \frac{\eps}{3} \phi_0 \\
    &\ge  \frac{c_2}{c_1} + \sqrt{\frac{c_{3,1}}{c_1}} + \sqrt{\frac{c_{3,2}}{c_1}} + \frac{2}{c_1 \tilde T} \\
    &\ge  \frac{c_2 + \sqrt{c_1 (c_{3,1} + c_{3,2})}}{c_1} + \frac{2}{c_1 \tilde T},
  \end{align*}
  where we have again set $\phi_0 \defs \phi_0(s_0)$,
  Lemma~\ref{lm:ode_blow_up} would imply $\tilde T \le \frac12 \tilde T$,
  hence our assumption that $\tmax = \infty$ must be false.

  Finally, \eqref{eq:blow_up:limsup_u} is a direct consequence of Lemma~\ref{lm:local_ex}.
\end{proof}

\section{Notes on global solvability}
Finally, we include short proofs for Proposition~\ref{prop:critical_mass} and Proposition~\ref{prop:global_ex}.

\begin{proof}[Proof of Proposition~\ref{prop:critical_mass}]
  This proof is based on a comparison principle for the scalar equation \eqref{lm:mass}.
  A similar idea (with a similar supersolution) has been employed in \cite[Lemma~5.2]{TaoWinklerCriticalMassInfinitetime2017}.

  Let $u_0 \in \con0$ be nonnegative and radially symmetric with $m_0 \defs \intom u_0 \lt 8 \pi$ as well as
  $(u, v)$ and $w$ be as constructed in Lemma~\ref{lm:local_ex} and defined in Lemma~\ref{lm:pde_w}, respectively.
  Then we may choose $a \in (\frac{m_0}{2\pi}, 4)$ and
  as $w(\cdot, 0) \le \frac{m_0}{2\pi}$ in $(0, R^2)$ and
  \begin{align*}
    \lim_{b \sea 0} \sup_{s \in (0, R^2)} \left| \frac{a s}{b + s} - a \right| = 0
  \end{align*}
  we may also choose $b \gt 0$ such that
  \begin{align*}
    \ol w: [0, R^2] \times [0, \infty) \ra \R, \quad
    (s, t) \mapsto \frac{a s}{b + s}
  \end{align*}
  fulfills $\ol w(\cdot, 0) \ge w(\cdot, 0)$ in $(0, R^2)$.

  Furthermore, by a direct computation
  \begin{align*}
    \ol w_s(s, t) = \frac{ab}{(b + s)^2} 
    \quad \text{and} \quad
    \ol w_{ss}(s, t) = -\frac{2ab}{(b + s)^3}
  \end{align*}
  for all $(s, t) \in (0, R^2) \times (0, \infty)$,
  hence
  \begin{align*}
    &\pe  \ol w_t(s, t)
          - 4s \ol w_{ss}(s, t)
          - 2 \ol w(s, t) \ol w_s(s, t)
          + \frac{m(t)}{|\Omega|} s \ol w_s(s, t)
          + 2^{p-1} \int_0^s 2^{p-1} \mu(\sqrt \sigma) \ol w_s^p(\sigma, t) \dsigma \\
    &\ge  \frac{8 a b s}{(b + s)^3} - \frac{2 a^2 b s}{(b + s)^3}
    \ge   0
  \end{align*}
  for all $(s, t) \in (0, R^2) \times (0, \infty)$ because of $a \le 4$.

  Therefore, $\ol w$ is a supersolution of \eqref{eq:pde_w:pde},
  fulfills $\ol w(0, \cdot) \ge 0$ and $\ol w(R^2, \cdot) \ge \frac{m_0}{2\pi}$
  as well as $\ol w(\cdot, 0) \ge w(\cdot, 0)$
  such that the comparison principle warrants $\ol w \ge w$ in $(0, R^2) \times (0, \tmax)$.

  As $w(0, t) = \ol w(0, t) = 0$ for all $t \in (0, \tmax)$ this implies
  \begin{align*}
        \limsup_{t \nea \tmax} w_s(0, t)
    =   \limsup_{t \nea \tmax} \lim_{h \sea 0} \frac{w(h, t) - w(0, t)}{h}
    \le \limsup_{t \nea \tmax} \lim_{h \sea 0} \frac{\ol w(h, t) - \ol w(0, t)}{h}
    \lt \infty.
  \end{align*}
  
  Due to non-degeneracy of \eqref{eq:pde_w:pde} outside of the origin and boundedness of $w$ parabolic regularity ensures
  $\limsup_{t \nea \tmax} \|w_s(\cdot, t)\|_{L^\infty((0, R^2))} \lt \infty$
  implying $\tmax = \infty$ by Lemma~\ref{lm:local_ex}.
\end{proof}

\begin{proof}[Proof of Proposition~\ref{prop:global_ex}]
  Let $p \gt 2$, $\alpha \lt p - 2$, $\mu_1 \gt 0$, $\kappa, \mu \in C^1([0, R])$ be such that \eqref{eq:global_ex:cond_mu} holds,
  $0 \le u_0 \in \con 0$
  and denote the corresponding solution given by Lemma~\ref{lm:local_ex} by $(u, v)$.

  By our assumption on $\alpha$ there exists $q \in (1, \min\{\frac{2p - 4 - \alpha}{\alpha}, 2\})$.
  Testing the first equation with $u^{q-1}$ gives
  \begin{align*}
      \frac1q \ddt \intom u^q
    = - \frac{4(q-1)}{q^2} \intom |\nabla u^\frac{q}{2}|^2
      + \frac{q-1}{q} \intom \nabla u^q \cdot \nabla v
      + \intom \kappa u^q
      - \intom \mu u^{p+q-1}
  \end{align*}
  in $(0, \tmax)$,
  wherein
  \begin{align*}
          \intom \nabla u^q \cdot \nabla v
    &=    \intom u^{q+1} - \intom u^q \frac{m(t)}{|\Omega|}
     \le  \frac{q}{\mu_1 (q-1)} \intom |x|^\alpha u^{p+q-1}
          + c_1 \int_0^R r^{1 - \alpha \frac{q+1}{p-2}} \dr
  \end{align*}
  in $(0, \tmax)$ for some $c_1 \gt 0$ by Young's inequality (with exponents $\frac{p+q-1}{q+1}$ and $\frac{p+q-1}{p-2}$).

  By the definition of $q$ we have $1 - \alpha \frac{q+1}{p-2} \gt -1$,
  hence the function $y: [0, \tmax) \ra \R, t \mapsto \frac1q \intom u^q$ satisfies
  $y' \le c_2$ in $(0, \tmax)$ for some $c_2 \gt 0$.

  Assuming for the sake of contradiction $\tmax \lt \infty$,
  this implies $\sup_{t \in (0, \tmax)} \intom u^q(\cdot, t) \lt \infty$,
  hence by elliptic regularity theory (cf.\ \cite[Theorem~19.1]{FriedmanPartialDifferentialEquations1976}) $\sup_{t \in (0, \tmax)} \|v(\cdot, t)\|_{\sob 2q}$ is finite as well.
  Therefore, the Sobolev embedding theorem warrants finiteness of $\sup_{t \in (0, \tmax)} \intom |\nabla v(\cdot, t)|^\frac{2q}{2-q}$.
  Finally, as $\frac{2q}{2-q} \gt 2$, a semi-group argument as in \cite[Lemma~4.1]{FuestBoundednessEnforcedMildly2019}
  shows boundedness of $\{u(\cdot, t): t \in (0, \tmax)\}$ in $\leb\infty$ -- contradicting Lemma~\ref{lm:local_ex}.
\end{proof}

\section*{Acknowledgments}
\small The author is partially supported by the German Academic Scholarship Foundation
and by the Deutsche Forschungsgemeinschaft within the project \emph{Emergence of structures and advantages in
cross-diffusion systems}, project number 411007140.

\newcommand{\noopsort}[1]{}

\end{document}